\documentclass[11pt]{amsart}
\usepackage{amsmath,amsthm,amssymb,graphicx,fancyhdr}
\usepackage{natbib}

\newcommand{\eid}{\stackrel{d}{=}}
\newcommand{\one}{{\bf 1}}

\newcommand{\reals}{{\mathbb R}}
\newcommand{\bbr}{\reals}
\newcommand{\vep}{\varepsilon}

\newcommand{\bz}{{\bf z}}
\newcommand{\bw}{{\bf w}}

\newcommand{\btheta}{{\boldsymbol \theta}}

\newcommand{\BX}{{\bf X}}

\newcommand{\cP}{{\mathcal P}}

\newtheorem{theorem}{Theorem}[section]

\newtheorem{remark}{Remark}

   \newtheoremstyle{example}{\topsep}{\topsep}%
     {}
     {}
     {\bfseries}
     {}
     {\newline}
     {\thmname{#1}\thmnumber{ #2}\thmnote{ #3}}

   \theoremstyle{example}

\newtheorem*{example}{Example}

\def\Var{{\rm Var}}
\def\E{{\rm E}}
\def\P{{\rm P}}

\numberwithin{equation}{section}

\begin{document}

\title[High minima of non-smooth Gaussian processes]
{High minima of non-smooth Gaussian processes}

\author{Zhixin Wu}
\address{Shanghai Jiatong University \\
 800 Dongchuan RD \\
 Minhang District \\
 Shanghai, China}
\email{wuzhixin@sjtu.edu.cn}

\author{Arijit Chakrabarty}
\address{Theoretical Statistics and Mathematics Unit \\
Indian Statistical Institute \\
203 B.T. Road\\
Kolkata 700108, India}
\email{arijit.isi@gmail.com}

\author{Gennady Samorodnitsky}
\address{School of Operations Research and Information Engineering\\
and Department of Statistical Science \\
Cornell University \\
Ithaca, NY 14853}
\email{gs18@cornell.edu}

\thanks{Chakrabarty's research was partially supported by the MATRICS grant of the the Science and Engineering Research Board, Government of India.
Samorodnitsky's research   was partially supported by the ARO
grant  W911NF-18 -10318  at Cornell University.}

\subjclass{Primary 60G15, 60F10.  Secondary 60G70.}
\keywords{ Gaussian process, high excursions, minima
\vspace{.5ex}}

\begin{abstract}
In this short note we study the asymptotic behaviour of the minima
over compact intervals  of Gaussian processes, whose paths are not
necessarily smooth. We show that, beyond the logarithmic large deviation
Gaussian estimates, this problem is closely related to the
classical small-ball problem. Under certain conditions we estimate the
term describing the correction to the large deviation behaviour. In
addition, the asymptotic distribution of the location of the minimum,
conditionally on the minimum exceeding a high threshold, is also studied.
\end{abstract}

\maketitle

\section{Introduction}
\label{sec:intro}

Let $\BX =\bigl( X(t),\, t\in\bbr\bigr)$ be a centered Gaussian
process with continuous sample paths. For a compact subinterval
$[a,b]$ of the real line we are interested in the right tail of the
random variable $\min_{a\leq t\leq b}X(t)$. This is a complicated
object; see e.g.   \cite{guliashvili:tankov:2016} and
\cite{adler:moldavskaya:samorodnitsky:2014}. On the logarithmic scale,
however, this tail can be described as follows: 
\begin{equation} \label{e:log.tail}
\lim_{u\to\infty} \frac{1}{u^2} \log P\bigl( \min_{a\leq t\leq
  b}X(t)>u\bigr) = -\frac{1}{2\sigma_*^2(a,b)}\,,
\end{equation}
where
\begin{equation} \label{e:optimize}
  \sigma_*^2(a,b) =
  \min_{\nu\in M_1[a,b]}\int_{[a,b]}\int_{[a,b]}R_\BX(s,t)\,\nu(ds)\,\nu(dt)\,,
\end{equation}
with $R_\BX$ the covariance function of the process 
and   $M_1[a,b]$  the set of all Borel probability measures $\nu$
on $[a,b]$; see Theorem 5.1 in
\cite{adler:moldavskaya:samorodnitsky:2014}. The quantity in \eqref{e:optimize} is strictly positive whenever the tail probability in \eqref{e:log.tail} is strictly positive for  $u=0$. In order to obtain more
precise results on the right tail of the minimum than
\eqref{e:log.tail}, additional assumptions on the process $\BX$, in
addition to its continuity, are needed. In
\cite{chakrabarty:samorodnitsky:2018} such additional assumptions
guarantee that the process $\BX$ is very smooth. Under these
assumptions the optimization problem \eqref{e:optimize} has a unique
optimal solution, a probability measure $\nu_*$ whose support is a
finite set. If $k$ is the cardinality of that set, then (under a
non-degeneracy assumption),
\begin{equation} \label{e:power.corr}
  P\bigl( \min_{a\leq t\leq   b}X(t)>u\bigr)  \sim cu^{-k} \exp\left\{
    -\frac{1}{2\sigma_*^2(a,b)}u^2\right\}
\end{equation}
for some $c\in (0,\infty)$.

Our goal in this paper is to obtain results on the asymptotics of the
right tail of the Gaussian minimum, more precise than the logarithmic
asymptotics \eqref{e:log.tail}, when the process $\BX$ is not so
smooth as to satisfy the assumptions of
\cite{chakrabarty:samorodnitsky:2018} (and, hence, also
\eqref{e:power.corr}). Such more precise asymptotics are, clearly,
related to the support of the optimal measure in \eqref{e:optimize},
so the next Section \ref{sec:support} describes certain situations
where information on the optimal measure or, at least, on its support,
is available. The more precise asymptotic results on the tail of the
minima are in Section  \ref{sec:minima}; the results are the most
precise in the Markovian case. In Section \ref{sec:location} we show
that, in many cases, the law of the location of the minimum of a
non-smooth Gaussian process, given that the minimum is high,
converges, as the height of the minimum increases, to the minimizer in
the optimization problem \eqref{e:optimize}. We conclude with examples
in Section \ref{sec:examples}.

\section{The optimal measure and its support}
\label{sec:support}

When a Gaussian process is very smooth, optimal measures in the
optimization problem \eqref{e:optimize} are supported by finite sets;
see \cite{chakrabarty:samorodnitsky:2018}. On the other hand,
processes whose sample paths are sufficiently ``rough'' may lead to optimal
measures with large supports, For example, if $\BX$ is the stationary
Ornstein-Uhlenbeck process, with covariance function
$R_\BX(s,t)=\exp\{-|s-t|\}$, then the optimal measure in
\eqref{e:optimize}  is
$$
\nu_*=\frac{1}{2+b-a}\delta_{a}+\frac{1}{2+b-a}\delta_{b}+ \frac{b-a}{2+b-a}\lambda_{a,b}\,,
$$
where $\delta_x$ is a point mass at $x$, and $\lambda_{a,b}$ is the
uniform probability distribution on the interval $(a,b)$; see Example
6.2 in \cite{adler:moldavskaya:samorodnitsky:2014}. In this case the
optimal measure has a full support in the interval $[a,b]$. We now
demonstrate other situations where this phenomenon holds.

We start with considering certain stationary Gaussian processes, in
which case we will use the standard single variable notation for the
covariance function $R_\BX(t):=R_\BX(s,s+t)$, $s,t\in\bbr$. By
stationarity it is enough to take $a=0$ and consider intervals of the
type $[0,b]$, $b>0$. 

\begin{theorem} \label{t:convex.R}
Let $\BX =\bigl( X(t),\, t\in\bbr\bigr)$ be a centered stationary Gaussian
process with continuous sample paths and covariance function
$R_\BX$. Suppose that  $R_\BX$ is strictly convex on
$[0,b]$. Then the optimization problem \eqref{e:optimize} has a unique
optimal probability measure, which has a full support in the
interval $[0,b]$. 
\end{theorem}
\begin{proof}
By Polya's theorem, the spectral measure of the process $\BX$ has an 
absolutely continuous component which is of full support on
$\bbr$; see e.g. \cite{lukacs:1970}. 
Then there is a unique optimal probability measure 
$\nu_*$ in  the optimization problem  \eqref{e:optimize}; see
\cite{adler:moldavskaya:samorodnitsky:2014}. Furthermore, the strict
convexity of the covariance function implies that it is strictly
decreasing on $[0,b]$.

Note that the support of the optimal probability measure $\nu_*$
cannot consist of a single point, for in that case the value
of the double integral in \eqref{e:optimize} is $R_\BX(0)$, while any
two-point probability measure $\nu$ would give a strictly smaller integral. 
We show now that endpoints
$0$ and $b$ of the interval  belong to the support. By symmetry it is
enough to prove that $b$ is in the support of $\nu_*$.

Suppose that, to the contrary, for some $0<\vep<b$ we have
$\nu_*\bigl( [b-\vep,b]\bigr)=0$, and let $c$ be the right-most point
of the support of $\nu_*$. Then $0<c\leq b-\vep$. Choosing, if
necessary, a smaller $\vep$ we can assure that $c> \vep$ and that
$\nu_*\bigl( [0,c-\vep)\bigr)>0$. Construct now a new probability
measure, $\hat\nu_*$ by translating the positive mass of $\nu_*$ in the
interval $[c-\vep,c]$ to the interval $[b-\vep,b]$. By the strict
monotonicity of the covariance function,
\begin{align*}
  &\int_{[0,b]}\int_{[0,b]}R_\BX(t-s)\,\hat\nu_*(ds)\, \hat\nu_*(dt) \\
    = &\int_{[0,c-\vep)}\int_{[0,c-\vep)}R_\BX(t-s)\,\nu_*(ds)\, \nu_*(dt) \\
        + &\int_{[c-\vep,c]}\int_{[c-\vep,c]}R_\BX(t-s)\,\nu_*(ds)\, \nu_*(dt) \\
  +2&  \int_{[c-\vep,c]}\int_{[0,c-\vep)} R_\BX(b-c+t-s)\,\nu_*(ds)\,
      \nu_*(dt) \\
   < &\int_{[0,c-\vep)}\int_{[0,c-\vep)}R_\BX(t-s)\,\nu_*(ds)\, \nu_*(dt) \\
        +& \int_{[c-\vep,c]}\int_{[c-\vep,c]}R_\BX(t-s)\,\nu_*(ds)\, \nu_*(dt) \\
  +2&  \int_{[c-\vep,c]}\int_{[0,c-\vep)} R_\BX(t-s)\,\nu_*(ds)\,
      \nu_*(dt) \\
   &=\int_{[0,b]}\int_{[0,b]}R_\BX(t-s)\,\nu_*(ds)\, \nu_*(dt)\,,
\end{align*}
contradicting the optimality of the measure $\nu_*$.

Hence, the endpoints of the interval are in the support of $\nu_*$,
and we proceed to prove that the support of $\nu_*$ is the entire
interval $[0,b]$. Suppose that, to the contrary, there are points
$0\leq c_1<c_2\leq b$, both in the support of $\nu_*$, such that
$\nu_*\bigl( (c_1,c_2)\bigr)=0$. Denote
$$
m(t) = \int_{[0,b]}R_\BX(t-s)\, \nu_*(ds), \ 0\leq t\leq b\,.
$$
The optimality of the measure $\nu_*$ implies that $m(t)\geq
\sigma_*^2(0,b)$ (the optimal value of the double integral in
\eqref{e:optimize}) for all $0\leq t\leq b$, with equality 
on the support of 
$\nu_*$; see Theorem 4.3 in
\cite{adler:moldavskaya:samorodnitsky:2014}. Note that on the interval
$[c_1,c_2]$ this function, 
$$
m(t) = \int_{[0,c_1]}R_\BX(t-s)\, \nu_*(ds) +
\int_{[c_2,b]}R_\BX(s-t)\, \nu_*(ds)\,, 
$$
is strictly convex by the assumptions. Since
$m(c_1)=m(c_2)=\sigma_*^2(0,b)$, this rules out the possibility that 
$m(t)\geq \sigma_*^2(0,b)$ for $c_1<t<c_2$.  The resulting
contradiction completes the proof of the theorem. 
\end{proof}

For certain nonstationary Gaussian processes the optimization problem
\eqref{e:optimize}  can be explicitly solved. Here is one such
situation. Let $\bigl( B(t),\, t\geq 0\bigr)$ be the standard Brownian
motion, and $0<a<b<\infty$. Consider a centered Gaussian process of
the form
\begin{equation} \label{e:BM.mod}
  X(t) = \frac{1}{g(t)}B(t), \ a\leq t\leq b\,,
\end{equation}
where $g:\, [a,b]\to (0,\infty)$ is a 
continuous function.

\begin{theorem} \label{t:BM.mod}
(a) \ Suppose that $g$ is a nondecreasing concave and twice
continuously differentiable function on $[a,b]$. Define
$$
f(x) = -g(x)g^{\prime\prime}(x)\geq 0, \ a<x<b\,,
$$
$$
p_a= \frac{g(a)}{a}\bigl( g(a) - ag^\prime (a)\bigr)\geq 0\,,
$$
$$
p_b=g(b)g^\prime(b)\geq 0\,.
$$
Then the finite measure $\mu$ on $[a,b]$ defined by
\begin{equation} \label{e:mu}
\mu(dx) = p_a\delta_a(dx) + p_b\delta_b(dx)
+ f(x)\, dx, \ a\leq x\leq b\,,
\end{equation} 
is equal, up to a multiplicative constant, to an optimal
solution to the optimization problem
\eqref{e:optimize}.

(b) Suppose that $g$ is concave on $[a,b]$, and nondecreasing 
and twice continuously differentiable  on $[a_0,b]$, for some
$a< a_0<b$ such that $g(a_0)=a_0g^\prime(a_0)$. 
If $p_b$ is as in part (a), and
$$
f(x) = -g(x)g^{\prime\prime}(x)\geq 0, \ a_0<x<b\,,
$$
then the finite
measure  $\mu$ on $[a,b]$ defined by
\begin{equation} \label{e:mu1}
\mu(dx) =  p_b\delta_b(dx)
+ f(x)\, dx, \ a_0\leq x\leq b\,,
\end{equation} 
is equal, up to a multiplicative constant, to an optimal
solution to the optimization problem
\eqref{e:optimize}.
\end{theorem}
\begin{proof}
  Observe that the covariance function of the process $\BX$ is given by
  $$
  R_\BX(s,t) =\frac{s}{g(s)g(t)}, \ a\leq s\leq t\leq b\,.
  $$
  With the measure $\mu$ defined by \eqref{e:mu},
  \begin{align*}
    &\int_{[a,b]} R_\BX(s,t)\,\mu(ds)= p_aR_\BX(a,t) + p_bR_\BX(b,t) +
      \int_a^b  R_\BX(x,t)f(x)\, dx \\
=& \frac{g(a)}{a}\bigl( g(a) - ag^\prime (a)\bigr) \frac{a}{g(a)g(t)} 
+ g(b)g^\prime(b) \frac{t}{g(b)g(t)} \\
-&\int_a^t \frac{x}{g(t)g(x)}g(x)g^{\prime\prime}(x)\, dx
- \int_t^b \frac{t}{g(t)g(x)}g(x)g^{\prime\prime}(x)\, dx \\
=& \frac{1}{g(t)} \left[ g(a)-ag^\prime(a) + tg^\prime (b) 
- \int_a^t xg^{\prime\prime}(x)\, dx - t \int_t^b
  g^{\prime\prime}(x)\, dx\right] \\
=& \frac{1}{g(t)} \left[ g(a) + \int_a^t g^\prime(x)\, dx \right]
   =1
  \end{align*}
for each $a\leq t\leq b$. By Theorem 4.3 in
\cite{adler:moldavskaya:samorodnitsky:2014} this implies  
the claim of part (a).

For part (b) note that by the above argument we already know that
\begin{equation} \label{e:partial.claim}
\int_{[a,b]} R_\BX(s,t)\,\mu_1(ds) = 1
\end{equation} 
for all $a_0\leq t\leq b$. Appealing, once again, to Theorem 4.3 in
\cite{adler:moldavskaya:samorodnitsky:2014} we see that the claim of
part (b) will follow once we check that the value of the integral in
\eqref{e:partial.claim} is at least 1 for $a\leq t<a_0$. For such $t$,
\begin{align*}
 &\int_{[a,b]} R_\BX(s,t)\,\mu_1(ds) =   p_bR_\BX(b,t) + \int_{a_0}^b R_\BX(x,t)f(x)\, dx \\
=& \frac{1}{g(t)} \left[ tg^\prime (b) - t \int_{a_0}^b
  g^{\prime\prime}(x)\, dx\right] \\
=& \frac{tg^\prime(a_0)}{g(t)}\,.
\end{align*}
Since  by concavity of $g$,
$$
g(a_0)-g(t) = \int_t^{a_0}g^\prime(x)\, dx \geq
g^\prime(a_0)(a_0-t)\,, 
$$
we conclude that
\begin{align*}
g(t)\leq g(a_0) -a_0 g^\prime(a_0) +
tg^\prime(a_0)=tg^\prime(a_0)\,,
\end{align*}
which gives the required lower bound on the integral of the covariance
function. 
  \end{proof} 
  \begin{remark}{\rm 
It is clear that the assumption of continuous second derivative of
the function $g$ in Theorem \ref{t:BM.mod} can be replaced by
the assumption of absolutely continuous first derivative, in which case the
function $g^{\prime\prime}$ in the statement of the theorem is simply
a nonpositive derivative of $g^\prime$  in the sense of
absolute continuity. }
    \end{remark}

    \section{Tails of the minima}
    \label{sec:minima}

    In this section we describe certain situations in which we can
    give more precise asymptotics of the tail of the minimum of a
    Gaussian process $\BX$ beyond the logarithmic asymptotics in
    \eqref{e:log.tail}. In these situations the smoothness assumptions
    of \cite{chakrabarty:samorodnitsky:2018} are not satisfied, and
    asymptotics of the type \eqref{e:power.corr} are no longer
    applicable. Our most precise results apply to Gaussian Markov
    processes, of which the processes of the type defined in
    \eqref{e:BM.mod} are a special case.

    \begin{theorem} \label{t5}
Let $(X(t), \, a\le t\le b)$ be a centered Gaussian Markov process
with continuous sample paths, such that  an optimal measure $\nu_*$
in the optimization problem \eqref{e:optimize} has an absolutely
continuous component $\nu_{ac}$, whose density with  respect to the
Lebesgue measure has a version with 
\begin{equation} \label{e:eta}
\eta:=\inf_{x\in[a,b]}\frac{d\nu_{ac}(x)}{dx}>0\,.
\end{equation}
Then
\begin{eqnarray*}
-\infty&<&\liminf_{u\to\infty}u^{-2/3}\left(\log
           P\left(\min_{a\leq t\leq b}X(t)>u\right)+
           \frac1{2\sigma_*^2(a,b)}u^2\right) \\
&\le&\limsup_{u\to\infty}u^{-2/3}\left(\log
      P\left(\min_{a\leq t\leq b}X(t)>u\right)+
      \frac1{2\sigma_*^2(a,b)}u^2\right)<0\,. 
\end{eqnarray*}
\end{theorem}
\begin{proof}[Proof of Theorem \ref{t5}]
We will use the following easily checkable fact (which also follows
from  Theorem 4.12.11 (iii) of \cite{bingham:goldie:teugels:1987}): 
if $f:(0,\infty)\to(0,\infty)$ is a bounded measurable function such
that  
\begin{equation} \label{e:assume.f}
\lim_{\vep\downarrow0}\vep^{\beta}\log f(\vep)=-c\,,
\end{equation}
for some $\beta,c\in(0,\infty)$,
then there exists $C\in(0,\infty)$ such that 
\begin{equation} \label{e:conclude.f}
\lim_{x\to\infty}x^{-\beta/(1+\beta)}\log\int_0^\infty e^{-x\vep}f(\vep)d\vep=-C\,.
\end{equation}
 
Denote 
\begin{equation} \label{e:Y}
Y=\int_{[a,b]} X(t)\, \nu_*(dt)\,.
\end{equation} 
Since $\nu_*$ has full support, it follows that 
\[
\E(X(t)|Y)=Y\ \text{a.s. for all }t\in[a,b]\,;
\]
see e.g. p.8 in \cite{chakrabarty:samorodnitsky:2018}. With 
\[
Z(t):=X(t)-Y,\ t\in[a,b]\,,
\]
wee see that $Y$ and $\bigl(Z(t), \, t\in[a,b]\bigr)$ are
independent.  Since 
\[
\int_{[a,b]}Z(t)\,\nu_*(dt)=0\ \text{ a.s.}\,,
\]
it follows that 
\[
Z_*:=\min_{a\leq t\leq b}Z(t)\le0\ \text{ a.s.}\,.
\]
Therefore, for $u>0$,
\begin{align*}
&P\left(\min_{a\leq t\leq b}X(t)>u\right)
=P(Y+Z_*>u)\\
=&\int_u^\infty P(Z_*>u-y)\,P(Y\in dy)\\
=&\int_u^\infty P(Z_*>u-y)\frac1{\sigma_*(a,b)\sqrt{2\pi }}e^{-y^2/2\sigma_*^2(a,b)}\,dy\\
=&\frac1{\sigma_*(a,b)\sqrt{2\pi
   }}e^{-u^2/2\sigma_*^2(a,b)}\int_0^\infty
   e^{-u\vep/\sigma_*^2(a,b)}P(Z_*>-\vep)e^{-\vep^2/2\sigma_*^2(a,b)}\,d\vep\,.  
\end{align*}
We will prove that 
\begin{eqnarray}
\liminf_{\vep\downarrow0}\vep^2\log
  P(Z_*>-\vep)&>&-\infty\,,\label{gm.eq1} 
\end{eqnarray}
By \eqref{e:conclude.f}  with $\beta=2$ this will prove the lower
bound in the statement of the theorem.  However, if  $X_*$ and $X^*$
are the smallest and the largest values,
respectively,  of $\BX$ on $[a,b]$, then, as 
$\vep\downarrow0$, 
\begin{eqnarray*}
\log P(Z_*>-\vep)&\ge&\log P\left(X^*-X_*<\vep\right)
\sim-\kappa \vep^{-2}
\end{eqnarray*}
for some $\kappa\in(0,\infty)$. The asymptotic equivalence in the last
line has been shown in \cite{li:2001}. Thus, \eqref{gm.eq1} follows.

In order to prove the upper bound in the statement of the theorem, we
use a change of measure. Let ${\mathcal L}_\BX$ be the closed in $L^2$
linear span of the process $\BX$. For every $Z\in {\mathcal L}_\BX$,
the function $f_Z(t)= E\bigl( ZX(t), \, a\leq t\leq b\bigr)$ belongs
to the reproducing kernel Hilbert space of $\BX$ and, hence, the
probability measures $\bigl( X(t), \, a\leq t\leq b\bigr)$ and
$\bigl( X(t)+f_Z(t), \, a\leq t\leq b\bigr)$ generate on $\bbr^{[a,b]}$ are
equivalent. Furthermore, in the obvious notation, 
$$
\frac{dP^{\BX +{\bf f}_Z}}{dP^\BX} =\exp\left\{ Z-
  EZ^2/2\right\}\,; 
$$
see \cite{vandervaart:vanzanten:2008}. In particular, for every such
$Z$,
\begin{align} \label{e:RN.der}
  & P\bigl(\min_{a\leq t\leq b}X(t)+f_Z(t)>0\bigr) \\
 = & \exp\left\{ -  EZ^2/2\right\} E\left[ e^Z \one\bigl( \min_{a\leq
     t\leq b}X(t)>0 \bigr)\right] \,. \notag 
\end{align}
With $Y$ as in \eqref{e:Y} we choose $Z=-u Y/EY^2$. 
Since $\nu_*$ has a full support, we have $f_Z(t)=-u$ for all $a\leq
t\leq b$. By \eqref{e:RN.der},
\begin{align} \label{e:after.cm}
 P\bigl(&\min_{a\leq t\leq b}X(t)>u\bigr) 
 =  \exp\left\{ -  \frac1{2\sigma_*^2(a,b)}u^2\right\} \\
    & E\left[ \exp\left\{  -u\frac{1}{\sigma_*^2(a,b)}
      \int_{[a,b]}X(t)\, \nu_*(dt)\right\}  \one\bigl( \min_{a\leq
      t\leq b}X(t)>0\bigr)\right] \,. \notag 
 \end{align} 
Next,
\begin{align*}
& E\left[ \exp\left\{  -u\frac{1}{\sigma_*^2(a,b)}
      \int_{[a,b]}X(t)\, \nu_*(dt)\right\}  \one\bigl( \min_{a\leq
                 t\leq b}X(t)>0\bigr)\right] \\
  \leq &E\left[ \exp\left\{  -u\frac{1}{\sigma_*^2(a,b)}
         \int_{[a,b]}|X(t)|\, \nu_*(dt)\right\} \right] \\
  \leq & \exp\left\{  -u^{2/3}\frac{1}{\sigma_*^2(a,b)}\right\}
         + P\left( \int_{[a,b]}|X(t)|\, \nu_*(dt)\leq u^{-1/3}\right)\,.
\end{align*}
Appealing, once again, to \cite{li:2001}, we have, 
by \eqref{e:eta}\,,
\begin{align*}
  &\log P\left( \int_{[a,b]}|X(t)|\, \nu_*(dt)\leq u^{-1/3}\right) \\
  \leq & \log P\left( \eta \int_{[a,b]}|X(t)|\,  dt\leq u^{-1/3}\right) 
\sim-\kappa_1 u^{2/3}
\end{align*}
for some $\kappa_1\in(0,\infty)$. In conjunction with
\eqref{e:after.cm} this establishes the upper bound in the theorem. 
 \end{proof}

 It is clear from the proof of Theorem \ref{t5} that there is a close
 connection between the improvements on the logarithmic asymptotics
 \eqref{e:log.tail} of the minima of Gaussian processes and small
 ball problems for these processes. Availability of bounds on small
 ball probabilities is often helpful in obtaining bounds on the tail of
 the Gaussian minimum. The following theorem is another example of
 this.

 \begin{theorem}\label{t6}
Let $(X(t), \, a\le t\le b)$ be a centered Gaussian  process
with continuous sample paths, such that  an optimal measure $\nu_*$
in the optimization problem \eqref{e:optimize} has  a full support in
$[a,b]$. Suppose that there exists a function
$\sigma:[0,\infty)\to[0,\infty)$ satisfying
\begin{equation}
\label{eq5}\lim_{h\downarrow0}h^{-\beta}\sigma(h)=c\in(0,\infty)
\end{equation}
for some $\beta>0$, such that 
such that
\[
\E\left[(X(t)-X(s))^2\right]\leq \sigma(|t-s|)^2,\,s,t\in[a,b]\,.
\]
Then,
\[
\liminf_{u\to\infty}u^{-1/(\beta+1)}\left(\log
  P\left(\min_{t\in[a,b]}X(t)>u\right)
  +\frac1{2\sigma_*^2(a,b)}u^2\right)>-\infty\,,
\]
where $\sigma_*^2(a,b)$ is as in \eqref{e:optimize}, and should not be confused with the $\sigma$ of \eqref{eq5}.
\end{theorem}
\begin{proof}
 An argument identical to the proof of the lower bound in Theorem
 \ref{t5} gives us 
\begin{eqnarray*}
&&P\left(\min_{a\leq t\leq b}X(t)>u\right)\\
&\ge&\frac1{\sigma_*(a,b)\sqrt{2\pi}}e^{-u^2/2\sigma_*^2(a,b)} \\
  &&\int_0^\infty e^{-u\vep/\sigma_*^2(a,b)}
     P\left(\max_{a\leq t\leq  b}|X(t)-X(a)|<\vep/2\right)
     e^{-\vep^2/2\sigma_*^2(a,b)}\,d\vep\,.
\end{eqnarray*}
Since by the assumption \eqref{eq5} we have, for some $K\in
(0,\infty)$,  
 \[
   P\left(\max_{a\leq t\leq b}|X(t)-X(a)|\le\vep\right)\ge
   \exp\left(-K\vep^{-1/\beta}\right),\,\vep>0\,,
\]
by Theorem 4.1 in \cite{li:shao:2001}, the claim of the theorem
follows from \eqref{e:conclude.f}. 
\end{proof}

\section{The location of the minimum}
\label{sec:location}

For a continuous centered Gaussian process $\BX =\bigl( X(t),\,
t\in\bbr\bigr)$ consider the location of the minimum of the process on
an interval $[a,b]$: 
\[
T_{*}:=\arg\min_{a\leq t\leq b}X(t)\,,
\]
where we choose the leftmost location of the minimum in case there are
ties. For very smooth Gaussian processes considered in
\cite{chakrabarty:samorodnitsky:2018} it was proved that, as
$u\to\infty$, 
\begin{equation} \label{e:argmin.conv}
\P\left(T_{*}\in\cdot\, \Bigr|\min_{a\leq t\leq
    b}X(t)>u\right)\Rightarrow  \nu_*\,,
\end{equation}
with $\nu_*$  the unique minimizer in the optimization problem
\eqref{e:optimize}. In that case the latter optimal measure is always
supported by a finite set. Our goal in this section is to show
that \eqref{e:argmin.conv} continues to hold for Gaussian processes
whose sample paths are not smooth, and for which the optimal measure
may have full support. 
\begin{theorem} \label{t:zhixin}
Let $\BX =\bigl( X(t),\, t\in\bbr\bigr)$ be a centered stationary Gaussian
process with continuous sample paths and covariance function
$R_\BX$. Suppose that  $R_\BX$ is  strictly convex on
$[0,b]$. Then \eqref{e:argmin.conv} holds with $a=0$ and any $b>0$,
where $\nu_*$ is the unique optimal probability measure for the 
the optimization problem \eqref{e:optimize}.  
\end{theorem}
\begin{proof}
The fact that the optimization problem \eqref{e:optimize} has a
unique optimal solution $\nu_*$ was established in Theorem
\ref{t:convex.R}. We use \eqref{e:after.cm} (with $a=0$). Let
$A\subseteq [a,b]$ be a Borel set that is a continuity set for
$\nu_*$. Recalling the notation \eqref{e:Y} we obtain
\begin{align*} 
  & P\bigl( T_*\in A|\min_{a\leq t\leq b}X(t)>u\bigr) \\
  =
&\frac{E\left[ \exp\left\{  -u Y/\sigma_*^2(a,b)
      \right\}  \one\bigl( \min_{a\leq
      t\leq b}X(t)>0, \, T_*\in A\bigr)\right]}
{E\left[ \exp\left\{  -u Y/\sigma_*^2(a,b)
      \right\}  \one\bigl( \min_{a\leq
      t\leq b}X(t)>0\bigr)\right]}\,.
\end{align*}
By Fubini's theorem this can be rewritten in the form
\begin{align*} 
  & P\bigl( T_*\in A|\min_{a\leq t\leq b}X(t)>u\bigr) \\
  =
&\frac{\int_0^\infty \exp\left\{  -u x/\sigma_*^2(a,b)
      \right\}  P\bigl( Y\leq x,\, \min_{a\leq
      t\leq b}X(t)>0, \, T_*\in A\bigr)\, dx}
{\int_0^\infty \exp\left\{  -u x/\sigma_*^2(a,b)
      \right\}  P\bigl( Y\leq x,\, \min_{a\leq
      t\leq b}X(t)>0\bigr) \, dx}\,, 
\end{align*}
and so it is enough to prove that
\begin{align*} 
\nu_*(A) = &\lim_{x\to 0}\frac{P\bigl( Y\leq x,\, \min_{a\leq
      t\leq b}X(t)>0, \, T_*\in A\bigr)}{P\bigl( Y\leq x,\, \min_{a\leq
      t\leq b}X(t)>0\bigr)} \\
  = &\lim_{x\to 0} P\bigl(T_*\in A\big| Y\leq x,\, \min_{a\leq
    t\leq b}X(t)>0\bigr)\,. 
\end{align*}
If we denote by $m_x$ the probability measure described by the right
hand side of this statement, then we need to prove that
\begin{equation} \label{e:dual.conv}
  m_x\Rightarrow \nu_* \ \text{as $x\to 0$.}
  \end{equation}

To this end, we use a discrete approximation. Let $\cP_k=\bigl\{
bi2^{-k}, \, i=0,1,\ldots, 2^{k}\bigr\}$ be the $k$th binary partition
of the interval $[0,b]$, $k=1,2,\ldots$. For each $k$ we consider the  following
restricted version of the optimization problem \eqref{e:optimize}: 
\begin{equation} \label{e:optimize.k}
  \min_{\nu\in M_1(\cP_k)}\int_{[0,b]}\int_{[0,b]}R_\BX(s,t)\,\nu(ds)\,\nu(dt)\,,
\end{equation}
where the probability measures are required to be supported by the
finite set $\cP_k$. As in the case of the full optimization problem
\eqref{e:optimize}, the fact that the spectral measure of the process
$\BX$ is of full support guarantees that the problem
\eqref{e:optimize.k} has a unique optimal solution, which we will
denote by $\nu_{*,k}$. We also denote by $\sigma_{*,k}^2$ the
corresponding value of the double integral. 
The same argument as in the case of the
restricted optimization problem shows that, because of strict
convexity of $R_\BX$, $\nu_{*,k}$ assigns a positive mass to each
point in $\cP_k$. 

Clearly, $\sigma_{*,1}^2\geq \sigma_{*,2}^2\geq \ldots \geq
\sigma_{*}^2[0,b]$. On the other hand, the obvious discretizations of
the measure $\nu_*$ produce a sequence of probability measures
$\nu_k^\prime\in \cP_k$, $k=1,2,\ldots$ such that
$\nu_k^\prime\Rightarrow \nu_*$ as $k\to\infty$. By continuity,
$$
\int_{[0,b]}\int_{[0,b]}R_\BX(s,t)\,\nu_k^\prime(ds)\,\nu_k^\prime(dt)
\to \int_{[0,b]}\int_{[0,b]}R_\BX(s,t)\,\nu_*(ds)\,\nu_*(dt)
$$
as $k\to\infty$, so by the optimality of the measures $(\nu_{*,k})$
we conclude that $\sigma_{*,k}^2\to \sigma_{*}^2[0,b]$. We claim that
$\nu_{*,k}\Rightarrow \nu_*$. Since the space $M_1[0,b]$ is weakly
compact, it is enough to prove that every subsequential limit of the
sequence $(\nu_{*,k})$ is equal to $\nu_*$. However, for every
subsequence of of the sequence $(\nu_{*,k})$ the value of the double
integral in the optimization problem \eqref{e:optimize} converges to
$\sigma_{*}^2[0,b]$ and, by weak continuity of the double integral, it
also converges to the double integral with respect to the
subsequential limit. Since under the assumptions of the theorem the
optimization problem \eqref{e:optimize} has a unique optimal solution,
we conclude that  every subsequential limit of the
sequence $(\nu_{*,k})$ is equal to $\nu_*$.

Define, analogously to \eqref{e:Y},
$$
Y_k=\int_{[a,b]} X(t)\, \nu_{*,k}(dt)\,, 
$$
and let 
\[
T_{*,k}:=\arg\min_{t\in\cP_k}X(t) \, k=1,2,\ldots\,,
\]
once again choosing the leftmost location in the case of a tie. 
For each $k$ we define a probability measure on $[a,b]$ by 
$$
m_{x,k}(A)=  P\bigl(T_{*,k}\in A\big| Y_k\leq x,\,
\min_{t\in\cP_k}X(t)>0\bigr),  \ \text{$A$ Borel.}
$$
It is clear that $T_{*,k}\to T_{*}$
and $\min_{t\in\cP_k}X(t)\to \min_{a\leq t\leq b}X(t)$
a.s. Furthermore, $Y_k\to Y$ in $L^2$. Furthermore, the distribution of
$\min_{a\leq t\leq b}X(t)$ is atomless (see Lemma 1 in \cite{ylvisaker:1965}). We
conclude that, for each fixed $x>0$, $m_{x,k}\Rightarrow m_x$ as
$k\to\infty$. It follows that the claim \eqref{e:dual.conv} will
follow if we prove that
\begin{equation} \label{e:dual.conv.k}
  m_{x,k}\Rightarrow \nu_{*,k} \ \text{uniformly in $k$ as $x\to 0$.} 
  \end{equation}

Consider the zero mean Gaussian  random vector  $\BX^{(k)}=\bigl(
X(bi2^{-k}), \, 
i=0,1, \ldots, 2^k\bigr)$. Let $\Sigma_k$ denote its covariance
matrix. The uniqueness of the minimizing measure $\nu_{*,k}$ implies
that the vector $\BX^{(k)}$ has full support, so $\Sigma_k$ is
invertible. For any $j=0,1, \ldots, 2^k$ we can write
\begin{align} \label{e:two.int}
m_{x,k}\bigl(\{bj2^{-k}\}\bigr)=&\frac{P\bigl(\BX^{(k)}\in
  E_j(x)\bigr)}{\sum_{i=0}^{2^k} 
                                  P\bigl(\BX^{(k)}\in E_i(x)\bigr)}\\
  =& \frac{\int_{E_j(x)} \exp\bigl\{ -\bz^T\Sigma_k^{-1}\bz/2\bigr\}\, d\bz}
   {\sum_{i=0}^{2^k}\int_{E_i(x)} \exp\bigl\{
     -\bz^T\Sigma_k^{-1}\bz/2\bigr\}\, d\bz}\,, \notag 
\end{align}
where
$$
E_j(x) =\bigl\{ \bz\in (0,\infty)^{2^k+1}, \,  z_j<z_i, \, i\not= j, \,
\sum_{i=0}^{2^k+1} \nu_{*,k}\bigl(\{bi2^{-k}\}\bigr) z_i\leq
x\bigr\}\,,
$$
$j=0,1,\ldots, 2^k+1$. It is straightforward to compute that
$$
\int_{E_j(x)}   d\bz = \frac{x^{2^k+1}}{(2^k+1)!}
\frac{\nu_{*,k}\bigl(\{bj2^{-k}\}\bigr)}{\prod_{i=0}^{2^k+1}
  \nu_{*,k}\bigl(\{bi2^{-k}\}\bigr)}
  $$
  Therefore, if we prove that
  \begin{equation} \label{e:unif.quad}
    \bz^T\Sigma_k^{-1}\bz \to 0 \ \text{as $x\to 0$}
  \end{equation}
  uniformly on $\cup_{i=0}^{2^k+1}E_i(x)$, then we obtain uniform
  convergence in \eqref{e:dual.conv.k} (even in total variation). 

  To this end, let $\bw=\Sigma_k^{-1}\bz$, so that
  $$
  \bz^T\Sigma_k^{-1}\bz=\bw^T\Sigma_k\bw\,.
  $$
  Let $\btheta=\Sigma^{-1}\one$. The vector $\btheta$ is equal, up to
  a multiplicative scale, to the probability vector of the measure
  $\nu_{*,k}$; see \cite{chakrabarty:samorodnitsky:2018}. Therefore,
  \begin{align*}
    \|\btheta\|_1 = \one^T \btheta = \one^T \Sigma_k^{-1} \one = 
 \btheta^T\Sigma_k\btheta 
    = (\|\btheta\|_1)^2\sigma^2_{*,k}\,,
  \end{align*}
  so that
  $$
  \|\btheta\|_1  = \frac{1}{\sigma^2_{*,k}}\,.
  $$
      In particular,
      $$
      \bw^T\one = \bz^T\btheta \leq \|\btheta\|_1 x
      = \frac{x}{\sigma_{*,k}^2}
      $$
      on  $\cup_{i=0}^{2^k+1}E_i(x)$. We conclude that
      $$
      \bw^T\Sigma_k\bw \leq R_X(0) \bigl( \bw^T\one \bigr)^2
      = R(0) \frac{x^2}{\sigma_{*,k}^4}\,.
      $$
  Since $ \sigma_{*,k}^2\to  \sigma^2_*[0,b]>0$, for all $k$ large
  enough we have $\sigma_{*,k}^2\geq  \sigma^2_*[0,b]/2$, and we have
  obtained the desired uniform convergence, thus completing the
  proof. 
\end{proof}

\section{Examples} \label{sec:examples} 
In this section, the results in Sections \ref{sec:support} - \ref{sec:location} are applied to two examples. The first example illustrates applications of Theorems \ref{t:BM.mod} and \ref{t5}.

\begin{example}{\rm
Let $(B(t):t\ge0)$ be a standard Brownian motion, and fix $0<\alpha<1$. Define
\[
X(t)=t^{-\alpha}B(t),\,t>0\,.
\]
Fix $0<a<b<\infty$, and set
\[
X_*=\min_{t\in[a,b]}X(t)\,.
\]
Theorem \ref{t:BM.mod} implies that the finite measure $\mu$ on $[a,b]$ defined by
\[
\mu_\alpha(dx)=\alpha(1-\alpha)x^{2\alpha-2}\,dx+(1-\alpha)a^{2\alpha-1}\delta_a(dx)+\alpha b^{2\alpha-1}\delta_b(dx)\,,
\]
is a constant multiple of the optimal measure, that is, the solution to the optimization problem \eqref{e:optimize}. Let
 \[
 \sigma_*^2(a,b;\alpha)=\mu_\alpha([a,b])^{-2}\Var\left(\int_a^bX(t)\mu_\alpha(dt)\right)\,.
 \]
As the Radon-Nykodym derivative of the absolutely continuous component of $\mu_\alpha$ with respect to the Lebesgue measure is bounded away from $0$ on $[a,b]$, the hypotheses of Theorem \ref{t5} are clearly satisfied, which implies that
\[
\liminf_{u\to\infty}u^{-2/3}\left(\log P(X_*>u)+\frac1{2\sigma_*^2(a,b;\alpha)}u^2\right)>-\infty\,,
\]
and
\[
\limsup_{u\to\infty}u^{-2/3}\left(\log P(X_*>u)+\frac1{2\sigma_*^2(a,b;\alpha)}u^2\right)<0\,.
\]
In other words, as $u\to\infty$,
\begin{equation}
\label{eg1.eq1}P(X_*>u)=\exp\left(-\frac1{2\sigma_*^2(a,b;\alpha)}u^2-u^{\frac23+O(1/\log u)}\right)\,.
\end{equation}

When $\alpha=1/2$, $X(t)$ is a time-changed Ornstein-Uhlenbeck process. That is, 
\[
\left(X(e^{2t}):t\in\bbr\right)\eid\left(Z_t:t\in\bbr\right)\,,
\]
the process on the right hand side being an Ornstein-Uhlenbeck process. Therefore, a special case of \eqref{eg1.eq1} is that for any compact interval $[a,b]\subset\bbr$,
\begin{equation}
\label{eg1.eq2}P\left(\min_{t\in[a,b]}Z_t>u\right)=\exp\left(-\frac1{2\sigma_*^2(e^{2a},e^{2b};1/2)}u^2-u^{\frac23+O(1/\log u)}\right)\,,
\end{equation}
as $u\to\infty$.
}
\end{example}

The second example illustrates applications of Theorems \ref{t:convex.R}, \ref{t6} and \ref{t:zhixin}. 

\begin{example}
  {\rm
    Let $(X(t):t\in\bbr)$ be a stationary Gaussian process with mean zero and covariance function
    \[
    R_\BX(t)
    = \exp\bigl\{ -|t|^{\alpha}\bigr\},\,t\in\bbr\,,
    \]
for a fixed  $0<\alpha\leq 1$. The assumptions of Theorem
\ref{t:convex.R} are satisfied for any $b>0$ and, hence, the optimal
measure, say $\nu_*$, in the
optimization problem \eqref{e:optimize}  is of full support. If
$\alpha=1$, this follows from the explicit solution of the optimization
problem in \cite{adler:moldavskaya:samorodnitsky:2014}. 

The hypotheses of Theorem \ref{t6} are therefore satisfied with $\beta=\alpha/2$ and $c=\sqrt2$, which implies the existence of $C\in(0,\infty)$ satisfying
\[
\log P\left(\min_{t\in[a,b]}X(t)>u\right)\ge-\frac1{2\sigma_*^2(a,b)}u^2-Cu^{-2/(\alpha+2)}\,,
\]
for large $u$. When $\alpha=1$ this reduces to the upper bound in
\eqref{eg1.eq2}. 

Finally, an appeal to Theorem \ref{t:zhixin} shows that the conditional law of the location of the minimum (the leftmost one to be chosen in case of ties) on $[a,b]$ given that the minimum if above $u$, converges weakly to $\nu_*$ as $u\to\infty$.
  }
  \end{example}

\end{document}